\documentclass{birkau}
\usepackage{graphicx}
\usepackage{newlattice}

\theoremstyle{plain}

\newtheorem*{proposition}{Proposition}

\theoremstyle{definition}

\begin{document}
\title[Congruence lattices of SPS lattices]{Notes on planar semimodular lattices. VIII.\\ Congruence lattices of SPS lattices}
\author{G. Gr\"{a}tzer} 
\email{gratzer@me.com}
\urladdr{http://server.maths.umanitoba.ca/homepages/gratzer/}
\date{Aug. 31, 2019}

\begin{abstract}
In this note, 
I find a new property of the congruence lattice, $\Con L$,
of an SPS lattice $L$ (slim, planar, semimodular,
where ``slim'' is the absence of~$\mathsf M_3$ sublattices) with more than $2$ elements:
\emph{there are at least two dual atoms in $\Con L$}.
So the three-element chain 
cannot be represented as the congruence lattice
of an SPS lattice, supplementing a~result of G. Cz\'edli.
\end{abstract}

\maketitle

\section{Introduction}\label{S:intro}

My paper \cite{gG16}
investigates congruences of
fork extensions and in \cite[Theorem~5]{gG16} 
I prove that the congruence lattice, $\Con L$,
of an SPS lattice (a Slim Planar Semimodular lattice,
where ``slim''means that there is no $\SM 3$ sublattice)
has the following property:
\begin{enumeratei}
\item[(P1)] Every join-irreducible 
congruence has at most two join-irreducible covers 
(in the order of join-irreducible congruences).
\end{enumeratei}
G. Cz\'edli \cite[Theorem 1.1]{gC14} 
proves that the converse is false 
by exhibiting an eight-element 
distributive lattice, $\SD 8$, 
satisfying (P1), which 
cannot be represented
as the congruence lattice of an SPS lattice;
see also my paper \cite{gG15}.

In this note, I observe that
the congruence lattice of an SPS lattice $L$
with more than $2$ elements
has another property: 
\begin{enumeratei}
\item[(P2)] There are at least two dual atoms 
in $\Con L$.
\end{enumeratei}
So the three-element chain $\SC3$,
which satisfies (P1), cannot be represented
as the congruence lattice of~an SPS lattice,  
an example of minimal size.
Similarly, we can take any distributive lattice 
of more than two elements with a~single dual atom
to get a nonrepresentable lattice.

Note that G. Cz\'edli's eight-element example 
satisfies both (P1) and (P2).
To get a condition for representability, 
we have to go beyond these conditions.

For the basic concepts and notation,
see my paper \cite{gG16} or my books 
\cite{LTF} and \cite{CFL2}.

\section{Verification}\label{S:Verification}
In this section, we verify that 
the three-element chain $\SC3$ cannot be represented
as the congruence lattice of~an SPS lattice.

Following G.~Gr\"atzer and E.~Knapp~\cite{GKn07},
a planar semimodular lattice $L$ is \emph{rectangular},
if~the left boundary chain
has exactly one doubly-irreducible element, $c_l$,
and the right boundary chain  
has exactly one doubly-irreducible element, $c_r$,
and these elements are complementary.

Let us assume that $L$ is an SPS lattice 
with $\Con L \iso \SC3$.
The following statement is a contradiction.

\begin{proposition}\label{P:primes}
Let $L$ be an SPS lattice with more than $2$ elements. 
Then \lp P2\rp holds.
\end{proposition}

\begin{proof}
By G. Gr\"atzer and E. Knapp \cite[Theorem 7]{GKn09}, 
we can assume that~$L$ is rectangular.
We prove the following stronger version 
of the Proposition for slim rectangular lattices:

\emph{Let $L$ be a slim rectangular lattice 
with $c_l, c_t \neq 1$.
Then $P_l = [0, c_l]$ and $P_r  = [0, c_r]$ 
are distinct prime ideals of $L$.}

By G.~Cz\'edli and E.\,T. Schmidt~\cite[Lemma 22]{CS12},
we can obtain $L$ from a~grid $D = \SC p \times \SC q$,
where $p,q \geq 2$,
with a sequence $D = L_1, L_2, \dots, L_n = L$
of slim rectangular lattices,
where we obtain $L_{i+1}$ from $L_{i}$ 
by inserting a~fork 
at the covering square 
$S^i = \set{o^i, a_l^i, a_r^i, t^i}$ 
for $i = 1, \dots, n-1$, using the standard notation
\cite[Section 3, and Figure 5]{gG16}. 

\begin{figure}[htb]
\centerline{\includegraphics[scale = .8]{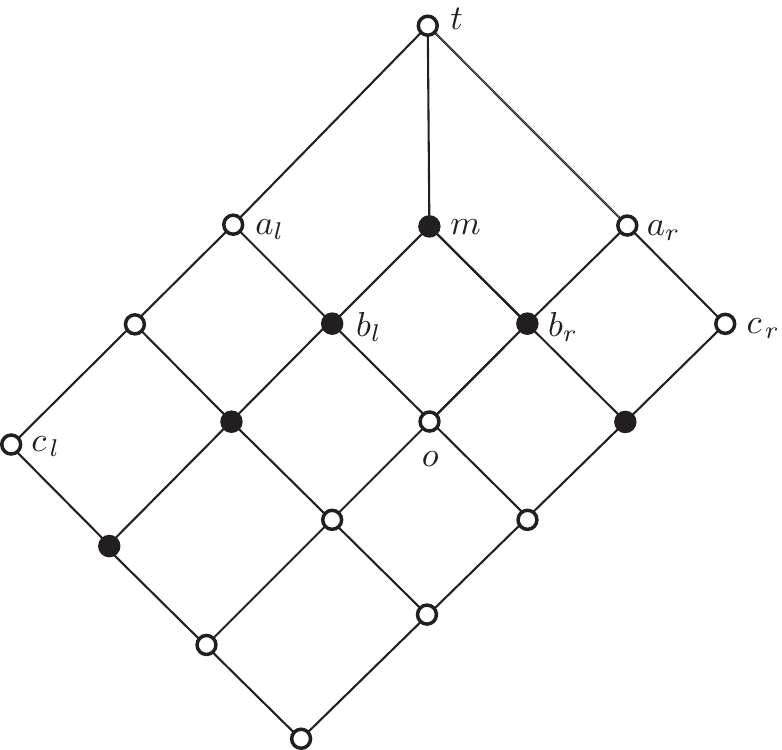}}
\caption{Fork insertion}
\label{F:fork}
\end{figure}

This statement holds for $i = 1$ by inspection,
see Figure~\ref{F:fork}. 
By~induction, assume that 
the statement holds for $i = n - 1$:
$P_l^i = [0, c_l^i]_{L_{n-1}}$ 
and $P_r^i = [0, c_r^i]_{L_{n-1}}$ 
are distinct prime ideals of $L_{n-1}$.
The induction step is the same as the induction base,
\emph{mutatis mutandis}.
\end{proof}

\end{document}